\documentclass{amsart}
\usepackage{color}
\usepackage[light, first]{draftcopy}
\usepackage{amsfonts}
\usepackage{amssymb}
\usepackage{amsthm}
\usepackage{newlfont}
\usepackage{graphicx}

\newtheorem{thm}{Theorem}[section]

\newtheorem{lem}[thm]{Lemma}
\newtheorem{prop}[thm]{Proposition}
\theoremstyle{definition}
\newtheorem{defn}[thm]{Definition}
\theoremstyle{remark}
\newtheorem{rem}[thm]{Remark}

\numberwithin{equation}{section}
\numberwithin{thm}{section}

\newcommand{\lsm}{\lesssim}

\newcommand{\C}{{\mathbb{C}}}
\newcommand{\R}{{\mathbb{R}}}
\renewcommand{\l}{\lambda}
\newcommand{\F}{{\mathcal F}}

\newcommand{\cd}{\cdot\,}

\newcommand{\ed}{\end {document}}

\newcounter{smalllist}

\title[energy-critical NLS]
{Global well-posedness and scattering for defocusing energy-critical NLS in the
exterior of balls with radial data}

\author[D. Li]{Dong Li}
\address{Department of Mathematics, University of Iowa, 14 MacLean Hall, Iowa City, USA 52242}%
\email{mpdongli@gmail.com}

\author[H. Smith]{Hart Smith}
\address{Department of Mathematics, University of Washington, Seattle, WA  98195-4350}
\email{hart@math.washington.edu}

\author[X. Zhang]{Xiaoyi Zhang}
\address{Department of Mathematics, University of Iowa, 14 MacLean Hall, Iowa City, USA 52242 and
Chinese Academy of Science, Beijing}%
\email{zh.xiaoyi@gmail.com}
{\normalsize }

\begin{document}
\maketitle
\begin{abstract}
We consider the defocusing energy-critical NLS in the exterior of
the unit ball in three dimensions. For the initial value problem
with Dirichlet boundary condition we prove global well-posedness and
scattering with large radial initial data in the Sobolev space $\dot
H_0^1$. We also point out that the same strategy can be used to
treat the energy-supercritical NLS in the exterior of balls with
Dirichlet boundary condition and radial $\dot H_0^1$ initial data.
\end{abstract}

\section{Introduction}
Let $\Omega=\R^3\setminus {\bar B}(0,1)$ be the exterior of the unit ball. We consider the defocusing energy critical NLS in $\Omega$ with Dirichlet boundary condition:
\begin{align}\label{nls}
\begin{cases}
i\partial_t u+\Delta u=|u|^4 u\equiv F(u)\,, \ (t,x)\in \R\times \Omega,\\
u(t,x)|_{\R \times \partial \Omega}=0,\\
u(0,x)=u_0(x).
\end{cases}
\end{align}

Our main purpose is to prove the global solvability and scattering
for the solution to \eqref{nls} under the assumption that $u_0\in
\dot H_0^1(\Omega)$ (see Section \ref{subsec_sob} for the
definition), and that $u_0$ is spherically symmetric.

In the whole space case $\R^n$ with $n\ge 3$, the Cauchy problem for
the energy critical NLS has been successfully attacked in both
defocusing and focusing cases \cite{borg:scatter, ckstt:gwp,
merlekenig, kv, RV, tao:radial, thesis:art}. On the other hand, the
understanding of the critical nonlinear problem of NLS posed on
exterior domains is still unsatisfactory. The difficulty comes from
several aspects. First of all, concerning linear estimates, the
dispersive estimates and Strichartz estimates are not always
available and often more limited than the whole space case.
Secondly, the nonlinear problem no longer has translation invariance
or scale invariance, and many of the technical tools built on
frequency analysis are not immediately applicable in the obstacle
case.

The Strichartz estimates on exterior domains or more general
Riemannian manifolds are usually obtained by using local smoothing
estimates \cite{bss:schrodinger, CS88, Sjolin87, Vega88} combined
with semi-classical parametrix constructions. For the domain
exterior to a non-trapping obstacle in $\R^n$, Blair, Smith and
Sogge \cite{bss:schrodinger} obtained a range of scale-invariant Strichartz
estimates, in particular the endpoint $L_t^4 L_x^\infty$ estimate in
dimension $n=3$, by using a
microlocal parametrix previously used for the wave equation in
\cite{bss:wave, ss:acta}. For the exterior domain to a strictly
convex obstacle, i.e. $\Omega=\R^n\setminus \mathcal K$, where
$\mathcal K$ is strictly convex, Ivanovici \cite{Iv08} obtained the
full range of Strichartz estimates excepting endpoints, by using the
Melrose-Taylor parametrix construction. For Strichartz estimates
with loss of derivatives, see \cite{BGT04, Anton08, Iva:ball}.

For the energy-critical nonlinear wave equation in 3 dimensional
smooth bounded domains with Dirichlet boundary condition,
Burq, Lebeau and Planchon \cite{BLP08} established global
well-posedness for $H_0^1$ solutions. Previously, Smith and Sogge \cite{SS10} proved
global well-posedness for the corresponding problem on the exterior domain to a
strictly convex obstacle.

\vspace{0.2cm}

In this paper, we shall prove the global well-posedness and
scattering of energy-critical NLS outside the unit
ball in $\R^3$ under the radial assumption. Our result is the following.

\begin{thm}\label{thm:main}
Let $\Omega=\R^3\setminus \bar{B}(0,1)$. Let $u_0\in \dot
H_{0}^1(\Omega)$ be spherically symmetric. Then there exists a
unique solution $u\in C_t^0\dot H_0^1(\mathbb R\times \Omega)$ to
\eqref{nls}, and
\begin{align*}
\|u\|_{L_{t,x}^{10}(\R\times\Omega)}\le C\bigl(\|u_0\|_{\dot
H_0^1(\Omega)}\bigr)\,.
\end{align*}
Moreover, there exist unique $v_{\pm}\in \dot H_0^1(\Omega)$ such
that
\begin{align*}
\lim_{t\to \pm\infty}\|u(t)-e^{it\Delta_D}v_{\pm}\|_{\dot H_0^1(\Omega)}=0.
\end{align*}
Here, $\Delta_D$ is the Dirichlet Laplacian and $e^{it\Delta_D}$ is the free propagator.
\end{thm}

\begin{rem}
The assumption $u_0 \in \dot H^1_0(\Omega)$ is very natural here due to the
energy critical nature of the problem. On the other hand, if we
assume $u_0 \in H^1_0(\Omega)$, then the proof can be trivialized,
see Section \ref{sec_inhomo} for details.
\end{rem}

The proof of Theorem \ref{thm:main} follows roughly the strategy in
the paper by Bourgain \cite{borg:scatter} for dimension $n=3,4$ and
Tao \cite{tao:radial} for all dimensions, which dealt with
defocusing energy critical NLS in the whole-space case with radial
data. However, since many technical tools are missing in this
setting, we devote a large portion of the work to establishing the
technical tools in analogy with the whole space case. A crucial fact
exploited in this paper is that under the radial assumption, the
eigenfunctions of the Dirichlet Laplacian in the domain exterior to
a ball can be explicitly computed. We then use this explicit
knowledge to establish the following basic estimates.

$\indent\bullet$ The fundamental solution is written explicitly
through spectral representation of Dirichlet Laplacian. The
$L^1-L^{\infty}$ dispersive estimate then follows from the explicit
representation of the linear solution.  As a consequence, we
prove the full range Strichartz estimates with no loss of
derivatives.

$\indent\bullet$ The Littlewood-Paley operators are defined through
functional calculus. Bernstein type estimates for the
Littlewood-Paley operators are also shown to hold true.

$\indent\bullet$ Sobolev spaces on the exterior domain $\dot
H_{0}^{1,p}(\Omega)$ and $\dot H_{D}^{1,p} (\Omega)$ for $1<p<3$ (see Subsection
\ref{subsec_sob}) are proved to be equivalent. Therefore the product
rule and chain rule for the Dirichlet Laplacian is still applicable
as in the whole space case.

With these technical tools in hand, we reduce Theorem \ref{thm:main}
to establishing the a priori bound of the $L_{t,x}^{10}$-norm of the
solution. Then we follow the spirit of the argument in
\cite{borg:scatter} and \cite{tao:radial}. Whilst a handful of
estimates still hold true as in the whole space case, the most
problematic part comes from proving the mass localization. In the
whole space case, the key fact used in the proof is that the
Littlewood-Paley operator is defined through convolution with a
normalized Schwartz function. This property no longer holds in our
setting since we do not have translation-invariance. Alternatively,
we shall prove the localization though a careful analysis on the
interaction between spatial and frequency cutoffs. A crucial result
is Lemma \ref{lem:lowerbound} where we show that any time interval
which carries nontrivial space-time norm must have a uniform lower
bound.\footnote{In the language of \cite{tao:radial}, we actually
showed that the length of unexceptional intervals must have a
uniform lower bound. This is quite different from the whole-space
case.}

In Section 2 we introduce basic notations and some useful estimates.
The proof of Theorem \ref{thm:main} is contained in Section 3. We
shall only focus on the parts which are different from the whole
space case: the mass localization and the Morawetz inequality. The
combinatorics argument, which eventually gives the upper bound of
the space-time norm of the solution, will only be sketched. One can
refer to \cite{borg:scatter} or \cite{tao:radial} for more details.
In section 4 and 5 we give remarks on the energy supercritical
problem and the case with inhomogeneous Sobolev data.

\section*{Acknowledgements}
Dong Li is supported by NSF Grant DMS-0908032. Hart Smith is supported by
NSF Grant DMS-0654415. Xiaoyi Zhang is supported by an Alfred P.
Sloan Research Fellowship and also start-up funding from the University of
Iowa.

\section{Basic estimates}

\subsection{Some notation}
We write $X \lesssim Y$ or $Y \gtrsim X$ to indicate $X \leq CY$ for
some non-essential constant $C>0$.  We use ${\mathcal O}(Y)$ to denote any
quantity $X$ such that $|X| \lesssim Y$.  We use the notation $X
\sim Y$ whenever $X \lesssim Y \lesssim X$.  The fact that these
constants depend upon the dimension $d$ will be suppressed.  If $C$
depends upon some additional parameters, we will indicate this with
subscripts; for example, $X \lesssim_u Y$ denotes the assertion that
$X \leq C_u Y$ for some $C_u$ depending on $u$. Sometimes when the
context is clear, we will suppress the dependence on $u$ and write
$X \lesssim_u Y$ as $X \lesssim Y$. We will write $C=C(Y_1, \cdots,
Y_n)$ to stress that the constant $C$ depends on quantities $Y_1$,
$\cdots$, $Y_n$.

Let $I\subset \mathbb R$ be a time interval.
 We write $L^q_t L^r_{x}
(I\times \Omega)$ to denote the Banach space with norm
$$ \| u \|_{L^q_t L^r_x(I \times \Omega)} :=
\biggl(\,\int_I \biggl(\,\int_{\Omega} |u(t,x)|^r\ dx\biggr)^{q/r}\
dt\biggr)^{1/q},$$ with the usual modifications when $q$ or $r$ are
equal to infinity. 
When $q=r$ we abbreviate $L^q_t L^q_x$ as $L^q_{t,x}$. We shall
write $u \in L_{t,loc}^q L_x^r (I\times \Omega)$ if  $u \in L_t^q
L_x^r (J\times \Omega)$ for any compact $J\subset I$.

For any positive number $1\le a\le \infty$, we let $a^\prime=a/(a-1)$ denote
the conjugate of $a$, so that $1/a + 1/a^\prime =1$.


We use $\mathcal S(\R)$ to denote the space of Schwartz functions, and $\mathcal S'(\R)$ the
space of tempered distributions, on the real line.



\subsection{Fundamental solution and Strichartz estimates}
The spectral resolution for radial functions on the exterior domain $r\ge 1$ in $\R^3$
is expressed
using the radial eigenfunctions
\begin{align*}
\Delta \phi_\l + \lambda^2 \phi_\l =0
\end{align*}
for $\l>0$
which satisfy the Sommerfeld radiation condition,
namely
\begin{align*}
\phi_{\lambda}(r) = \frac {\sin \lambda (r-1)} {r}, \quad r\ge 1.
\end{align*}
For tempered distributions $f\in{\mathcal S}'(\R)$, 
we set
$$
\F f(\l)=\frac 1{\sqrt{\pi}}\int \frac {\sin \lambda (s-1)} {s}\,f(s)\,s^2ds
$$
which can be expressed in terms of the Fourier transform of $s f(s)$
to identify $\F f$ as an odd element of ${\mathcal S}'(\R)\,.$
If $f\in {\mathcal S}(\R)$, then $\F f$ is a Schwartz function of $\l$.

We observe the following resolution of identity:
\begin{align*}
 \frac 1 \pi \int_{-\infty}^\infty \phi_{\lambda} (r) \phi_{\lambda}(s)\, d\lambda
= & \;\frac 1 {2\pi r s} \int_{-\infty}^\infty \cos \l(r-s)-\cos\l(r+s-2) \; d\lambda \\
= & \;\frac{\delta (r-s) +\delta (r+s -2) } {rs} \\
= & \;\frac{\delta (r-s)} {s^2}\,, \quad \text{for}\; r,s >1\,,
\end{align*}
from which it follows by a limiting procedure
that $\F^* \F f=f$ for $f\in {\mathcal S}$ supported in $[1,\infty)\,,$
where  $\F^*$ is the formal adjoint, defined on
tempered distributions $g$ as the restriction to $r\ne 0$ of
$$
\F^*g=\frac 1{\sqrt{\pi}}\int_{-\infty}^\infty \frac{\sin \l(r-1)}r\,g(\l)\,d\l\,.
$$
Consequently, $f\rightarrow \F_0f=\sqrt{2}\,\F f|_{\l>0}$ induces an isometric map
\begin{equation}\label{plancherel}
\F_0\,:\,L^2([1,\infty),s^2ds)\rightarrow L^2([0,\infty),d\l)\,.
\end{equation}
One can similarly verify that if $g$ is an odd element of ${\mathcal S}(\R)$, then
\begin{equation}\label{plancherel2}
{ \frac 2 \pi}
\int_1^\infty \int_0^\infty \frac{\sin\l(r-1)}r\,\frac{\sin\mu(r-1)}r\,g(\mu)\,d\mu\,r^2dr=g(\l)\,,
\end{equation}
hence $\F_0$ in \eqref{plancherel} is onto, and thus an isomorphism of Hilbert spaces.

We will also use the radial inhomogeneous Sobolev space $\dot H^1_0(\Omega)$, defined as
the closure of $C_c^\infty([1,\infty))$ in the norm
$$
\|f\|_{\dot H^1_0}=\|f'(r)\|_{L^2(r^2dr)}\,.
$$
By Sobolev embedding on $\R^3$, for compactly supported $f$ we have
$$
\|f\|_{L^6(r^2dr)}\le \|f'\|_{L^2(r^2dr)}\,.
$$
Conversely, if $f\in L^6(r^2dr)$ and $f'\in L^2(r^2dr)$, then
$\chi(N^{-1} r) f$ converges to $f$ in the $\dot H^1_0$ norm, hence we may identify
$\dot H^1_0(\Omega)$ as absolutely continuous functions on $[1,\infty)$ for which
$$
\|f'\|_{L^2(r^2 dr)}+\|f\|_{L^6(r^2dr)}<\infty\,,\qquad f(1)=0\,.
$$

If $f\in C_c^\infty(\Omega)$ is radial, then $\F_0(\Delta f)(\l)=-\l^2\F_0 f(\l)\,.$
On the other hand,
$$
\int_\Omega \bar{f}\Delta f\,dx=\int_\Omega |\nabla f|^2\,dx=\int_\Omega |f'|^2\,dx\,,
$$
hence by \eqref{plancherel}, $\F_0$ induces an isometric map
$$
\F_0\,:\,\dot{H}^1_0(\Omega)\rightarrow L^2([0,\infty),\l\,d\l)\,.
$$
The image contains odd Schwartz functions by \eqref{plancherel2}, hence is an
isomorphism of Hilbert spaces.

Finally, we observe that
\begin{equation}\label{radial_emb}
|f(r)|\le r^{-\frac 12}\Bigl(\int_r^\infty |f'(s)|^2\,s^2ds\Bigr)^{\frac 12}\,,\quad
\text{hence}\quad \||x|^\frac 12 f\|_{L^\infty}\le\|f\|_{\dot H^1_0}\,.
\end{equation}
It follows as an easy consequence that radial
$\dot H^1_0$ is an algebra under multiplication of functions.

For $f\in L^2+\dot H^1_0$ (see Section 2.3 below), we can express
the Schrodinger propagator $e^{it\Delta_D}$ as
\begin{equation*}
(e^{it \Delta_D} f ) (r,t) = \F_0^*\bigl(e^{-i\l^2 t}\F_0 f\bigr)(r)\,.
\end{equation*}
The corresponding kernel is, where $t=\pm |t|$,
\begin{align*}
K(t,r,s)&=\frac 2\pi
\int_0^\infty \frac {\sin \lambda (r-1)} {r} \cdot
\frac{ \sin \lambda (s-1)} {s} \cdot e^{-i\lambda^2 t} \,d\lambda\\
&=\frac 1 {\pi rs}\int_{-\infty}^\infty \!\Bigl( e^{i\lambda(r-s)} - e^{i\lambda(r+s-2)} \Bigr) e^{-i\lambda^2 t} d\lambda\\
&= \frac {\pi^{\frac 12}e^{\pm i\pi/4}}{|t|^{\frac 12} r s}
 \left( e^{i(r-s)^2/4t} - e^{i (r+s-2)^2/4t } \right)\\
&= \frac {\pi^{\frac 12}e^{\pm i\pi/4}}{|t|^{\frac 12}rs}
 \left( 1-e^{i (r-1)(s-1)/t } \right)\,.
\end{align*}
It follows that $|K(t,r,s)|\le C\,|t|^{-3/2}$ for $r,s\ge 1$.
By a density argument, we thus have the important
\begin{lem}[Dispersive estimate] \label{lem_disp}
For $t\ne 0$ and radial $f \in L^2+\dot H^1_0$, we have
\begin{align*}
\| e^{it\Delta_D} f \|_{L^\infty (\Omega)} \lsm \frac 1 {|t|^{\frac
32}} \| f\|_{L^1(\Omega)}.
\end{align*}
\end{lem}

Strichartz estimates for radial data follow directly from this
dispersive estimate. See \cite{tao:keel} for instance. Therefore we
have the following lemma whose proof will be omitted.

%
%
%
\begin{lem}\label{strichartz}
Let $I$ be a time interval containing $0$. Let $u(t,x)$ satisfy
\begin{align*}
u(t,\cd) = e^{it \Delta_D} u_0
-i \int_0^t e^{i(t-s) \Delta_D} f(s,\cd) \,ds\,, \quad\forall\, t \in I\,,
\end{align*}
where $u_0\in L^2+\dot H^1_0\,, f\in L^1_t(L^2+\dot H^1_0)\,,$ with both radial.

Let $(q_i,r_i)$, $i=1,2$ be admissible pairs, such that $2\le
q_i\le\infty$, $\frac 2{q_i}+\frac 3{r_i}=\frac 32$. Then
\begin{align*}
\|u\|_{L_t^{q_1}L_x^{r_1}(I\times\Omega)}\lesssim
\|u_0\|_{L^2(\Omega)}+ \|f\|_{L_t^{q_2'}L_x^{r_2'}(I\times\Omega)}.
\end{align*}
Here $(q_2^\prime, r_2^\prime)$ are the conjugate exponents of
$(q_2,r_2)$.
\end{lem}

\subsection{Littlewood-Paley operators and Bernstein inequalities}
Given a bounded function $m(\l)$, which for convenience we assume
to be defined on all of $\R$ and even in $\l$, we define
$$
m\bigl(\sqrt{-\Delta_D}\,\bigr) f=
\F_0^*\bigl(m(\cdot)\,\F_0 f\bigr)\,.
$$
This defines a functional calculus on $L^2+\dot H^1_0$. In this section,
we will take $m$ to be an even function in $C_c^\infty(\R)$, in which case
$$
m\bigl(\sqrt{-\Delta_D}\,\bigr)f(r)=\int_1^\infty K_m(r,s)\,f(s)\,s^2ds\,,
$$
with
$$
K_m(r,s)=\frac 1\pi \cdot\frac{\widehat{m}(r-s)-\widehat{m}(r+s-2)}{rs}\,.
$$
In particular, we can define, for $N>0$, Littlewood-Paley projectors $P_N$ by taking $m=\psi(N^{-1}\lambda)$,
for suitable $\psi$ compactly supported away from 0. We similary define
$P_{\le N}$ using $m=\phi(N^{-1}\lambda)$, where $\phi\in C_c^\infty(\R)$ equals 1 on
a neighborhood of 0. We also set $P_{\ge N}=1-P_{\le N}$.

\begin{rem}
An added complication for this work, relative to the whole space case,
is that that spectral supports are not additive under function multiplication,
and thus we cannot exploit standard paraproduct decomposition results.


\end{rem}

As in the whole space case, we have the following

\begin{prop}[Bernstein inequality] \label{p:bernstein}
Let $1\le p\le q\le \infty$, and suppose $\sigma\in\R$. Then for any $N>0$
\begin{align}
\|P_{\le N} f\|_{L^q(\Omega)}\lesssim
&\; N^{3(\frac1p-\frac 1q)}
\|f\|_{L^p(\Omega)}\,,\label{bernstein1}\\
\|(-\Delta_D)^{\frac \sigma 2}P_Nf\|_{L^p(\Omega)}\approx
&\;N^\sigma \|P_N f\|_{L^p(\Omega)}\,.
\label{bernstein2}
\end{align}
\end{prop}
\begin{proof} We first prove \eqref{bernstein1}.
We write
$$
(P_{\le N}f)(r)=\int_1^{\infty}K_N(r,s)f(s)s^2ds\,,
$$
where
$$
K_N(r,s)=\frac N\pi\cdot\frac{\widehat\phi\bigl(N(r-s)\bigr)
-\widehat\phi\bigl(N(r+s-2)\bigr)}{rs} \,,
$$
and we observe that $\widehat\phi$ is an even Schwartz function.
Since $K$ is symmetric, it suffices by the Schur test and
interpolation to show that
\begin{align}
\sup_r\|K_N(r,s)\|_{L^1(s^2ds)}&\le C\,,\label{l1est}
\\
\sup_{r,s}|K_N(r,s)|&\le C\, N^3\,.\label{linfest}
\end{align}
We pose $r=1+N^{-1}x$, $s=1+N^{-1}y$, where $x,y>0$. Since $\frac sr\le\frac yx$, then
\eqref{l1est} is implied by the bound
\begin{equation}\label{l1est'}
\int_0^\infty \bigl|\,\widehat\phi(y+x)-\widehat\phi(y-x)\bigr|\,y\,dy
\le C\,x\,.
\end{equation}
For $y>2x$, we can bound
$|\widehat\phi(y+x)-\widehat\phi(y-x)|\lesssim x(1+y)^{-4}$,
which establishes \eqref{l1est'} for the integral over $y>2x$. We write
the remaining piece as
\begin{multline*}
\int_{-x}^x \bigl|\,\widehat\phi(y+2x)-\widehat\phi(y)\bigr|\,(y+x)\,dy\\
\le 2\bigl(\,\|\widehat\phi\|_{L^1(\R)}+\|y\,\widehat\phi(y)\|_{L^\infty(\R)}\bigr)\,x+
\int_{-x}^x|\widehat\phi(y+2x)|\,y\,dy\,.
\end{multline*}
To bound the final term, we use that $y+2x>x$ on the region of integration, to
see that $|\widehat\phi(y+2x)|\lesssim x^{-1}$, yielding a bound of $C\,x$ for this
term also.

For \eqref{linfest}, we
use the evenness of $\widehat\phi$ to write
$\widehat\phi(s)=g(s^2)$, where $g$ is Schwartz, to bound
$$
|K(r,s)|\le 4
N^3\left|\frac{g\bigl((x+y)^2\bigr)-g\bigl((x-y)^2\bigr)}{(x+y)^2-(x-y)^2}\right|\,.
$$
The inequality \eqref{linfest} follows, where we bound $|g'|\le C/4$.

Relation \eqref{bernstein2} follows similarly, by writing
$\l^\sigma\psi(N^{-1}\l)=N^\sigma \tilde\psi(N^{-1}\l)\psi(N^{-1}\l)$ where
$\tilde\psi\in C^\infty_c(\R)$, and applying \eqref{l1est} for the kernel associated
to $\tilde\psi(N^{-1}\l)$.
\end{proof}

\subsection{$L^p$ based Sobolev spaces} \label{subsec_sob}
We will also have need to work with the inhomogeneous
Sobolev norm for radial functions,
$\|f\|_{\dot H_0^{1,p}(\Omega)}=\|\nabla f\|_{L^p(\Omega)}=\|f'\|_{L^p(s^2ds)}$.
The difficulty with using this norm is that $\nabla$ does not commute with
$e^{it\Delta_D}$. This problem is solved by proving an equivalence
$$
\|f\|_{\dot H_0^{1,p}(\Omega)}\approx
\|(-\Delta_D)^{\frac 12}f\|_{L^p(\Omega)}\,.
$$
If $p>3$ this cannot hold for all $f\in \dot H^{1,p}_0$, as seen by taking
$f=\frac{\sin\l(r-1)}{\l r}\,,$ where the right side tends to 0 as $\lambda\rightarrow 0$,
but the left side remains bounded below. However, we shall only need to apply this equivalence for $1<p<3$ and for
$f\in \dot H^{1,2}_0=\dot H^1_0$, for which it does hold.
Note that in this case, both $|\nabla f|=|f'|$ and $(-\Delta_D)^{\frac 12} f$ belong to
$L^2(\Omega)$, so both sides of the equivalence are well defined.

\begin{prop}\label{sobolev}
Let $1< p< 3$. Then there exists a constant $0<C_p<\infty$, such
that for any radial function $f\in \dot H^1_0$, we have
\begin{equation*}
C_p^{-1} \| \nabla f \|_{L^p(\Omega)}  \le
\| (-\Delta_D)^{\frac 12} f \|_{L^p(\Omega)}  \le C_p \,\| \nabla f \|_{L^p(\Omega)}\,.
\end{equation*}
\end{prop}

\begin{proof}
We first establish that
\begin{equation}\label{ineq1}
\| (-\Delta_D)^{\frac 12} f \|_{L^p(\Omega)} \lsm \| f'\|_{L^p(\Omega)}\,.
\end{equation}
We will establish this under the assumption that $f(r)\in C_c^\infty([1,\infty))$.
To establish it for general $f$, we take a sequence
$f_j\in C_c^\infty([1,\infty))$ with $\|f'_j-f'\|_{L^q(\Omega)}\le 2^{-j}$ for
both $q=2$ and $q=p$. It follows that
$\|(-\Delta_D)^{\frac 12}(f_j-f)\|_{L^2(\Omega)}\lesssim 2^{-j}$, hence
$$(-\Delta_D)^{\frac 12}f_j\rightarrow (-\Delta_D)^{\frac 12}f\quad
\text{pointwise}\;\,a.e.
$$
The result for general $f\in \dot H^1_0$ then follows by Fatou's lemma.

For $f\in C_c^\infty$ we write
\begin{multline*}
\Bigl( (-\Delta_D)^{\frac 12} f \Bigr)(r)
  = \frac 2\pi\int_0^\infty
\int_1^\infty \frac{\sin \lambda (r-1)} r \, \frac {\sin \lambda
(s-1)} {s}\,
 \lambda f(s)\,s^2 ds\,d\lambda \\
= \frac 2\pi \int_0^\infty \int_1^\infty \frac{\sin \lambda (r-1)} {r}
\Bigl( s \cos \lambda(s-1) -
\frac 1 {\lambda} \sin \lambda (s -1) \Bigr) f^\prime
(s)\, ds\,d\l\,.
\end{multline*}

By considering the limit of the truncated integrals over $\l$, we obtain
\begin{align*}
\Bigl( (-\Delta_D)^{\frac 12} f \Bigr)(r) =
-\frac 2\pi\int_1^\infty K_1(r,s) \,f^\prime (s)\, s^2ds\,,
\end{align*}
where
$$
K_1(r,s) = \frac{1}{rs} \Biggl(\frac 1{r+s-2} + \frac 1 {r-s} \Biggr)
+ \frac 1 {rs^2} \log \left| \frac{r-s}{r+s-2} \right|\,,
$$
and the first term is interpreted as a principal value integral.
We note that this kernel is only applied to functions $f'(s)$ of integral $0$, hence
we can add a function $k(r)s^{-2}$ to $K_1(r,s)$ without changing the result. This
will indeed be necessary for small $p$.

A similar computation, using the whole-space spectral decomposition for radial functions
$$
\Bigl( (-\Delta)^{\frac 12} f \Bigr)(r)
  = \frac 2\pi \int_0^\infty
\int_0^\infty \frac{\sin \lambda r} r \, \frac {\sin \lambda
s} {s}\,
 \lambda f(s)\,s^2 ds\,d\lambda\,,
$$
expresses $(-\Delta)^{\frac 12}f(r)=-\frac 2\pi\int_0^\infty K_0(r,s)\,f'(s)\,s^2ds$,
with the kernel
$$
K_0(r,s) = \frac{1}{rs} \Biggl( \frac 1{r+s} + \frac 1 {r-s} \Biggr)
+ \frac 1 {rs^2} \log \left| \frac{r-s} {r+s} \right|\,.
$$
This kernel is bounded on $L^p(s^2ds)\,,\;1<p<\infty$,
since\footnote{This can be verified directly by elementary computation. Alternatively,
the two operators must agree on $g$ of integral $0$, determining $K_0(r,s)$ up to
$k(r)s^{-2}$. Since $K_0$ decrease like $s^{-3}$ as $s\rightarrow\infty$,
then necessarily $k(r)=0$.} it represents the operator
$$
g\;\rightarrow\;\sum_{j=1}^3 R_j\Biggl(\frac {x_j}{|x|}\,g\Biggr)\,,
$$
where $R_j$ is the Riesz transform $\partial_{x_j}(-\Delta)^{-\frac 12}$ on $\R^3$.

We are thus reduced to proving $L^p([1,\infty),s^2ds)$ boundedness of the
kernel $K=K_1-K_0$, that is
\begin{equation}\label{kform}
K(r,s)=\frac{1}{rs} \left( \frac 1{r+s-2} -\frac 1{r+s} \right)
+ \frac 1 {rs^2} \log \left( \frac{r+s}{r+s-2} \right)\,,
\end{equation}
with the freedom to add $k(r)s^{-2}$ to $K(r,s)$. Note that both terms on the right
of \eqref{kform}
are non-negative, hence can be considered separately. The $L^p$ boundedness of the first
term is based on the bound, for $1\le p< \infty$,
$$
\left\|\frac{2}{rs(r+s-2)(r+s)} 
\right\|_{L^{p'}\!(s^2ds)}
\lesssim\;\;
\bigl[r^2(r-1)\bigr]^{-p}\,,\quad r>1\,.
$$
This shows that the corresponding operator is weak-type $(p,p)$ for $1\le p<\infty$,
hence strong-type $(p,p)$ for $1<p<\infty$.

For the second term on the right of \eqref{kform},
we note that if $1\le r\le 2$ then for each $1<p'<\infty$ the $L^{p'}\!(s^2ds)$
norm is bounded uniformly in $r$.

For $r>2$, we write the second term as
\begin{equation}\label{logapprox}
-\frac{1}{rs^2}\log\left(1-\frac 2{r+s}\right)=\frac 2{rs^2(r+s)}+
{\mathcal O}\left(  \frac 1{rs^2(r+s)^2}    \right)\,.
\end{equation}
The second term on the right hand side of \eqref{logapprox}
is bounded by the first term in $K(r,s)$ considered
above. If $1<p'< 3$, then we dominate the $L^{p'}$ norm of the first term on the right
of \eqref{logapprox} by
$$
\left\|\frac{1}{s(r+s)^2}\right\|_{L^{p'}\!((0,\infty),s^2ds)}=C\,r^{-\frac 3p}\,.
$$
This implies strong-type $(p,p)$ bounds
for the second term in \eqref{kform} if $\frac 32<p<\infty$.

To obtain $(p,p)$ bounds for remaining $p$, we consider $1<p<3$, and
subtract the kernel $2/r^2s^2$. Since this kernel is bounded in $L^{p'}\!(s^2 ds)$
for $p'>\frac 32$, it does not affect the above consideration for $1<r<2$.
For $r>2$, we are reduced to considering
$$
\frac 2{r^2s^2}-\frac 2{rs^2(r+s)}=\frac {2}{r^2s(r+s)}\le \frac 4{r(r+s)^2}\,.
$$
We conclude by observing that, for $p'>\frac 32$,
$$
\left\|\frac{1}{r(r+s)^2}\right\|_{L^{p'}\!((0,\infty),s^2ds)}=C\,r^{-\frac 3p}\,,
$$
which yields the strong-type $(p,p)$ bounds for $1<p<3$. Note that we have in fact
established \eqref{ineq1} for all $1<p<\infty$.

To show the reverse implication, for $1<p<3$,
\begin{equation}\label{otherway}
\| \nabla f \|_{L^p(\Omega)} \lsm \| (-\Delta_D)^{\frac 12} f \|_{L^p(\Omega)}\quad
\text{if}\;\;f\in\dot H^1_0\,,
\end{equation}
we consider $f_N=P_{\le N}f$.
By \eqref{bernstein1},
$$
\| (-\Delta_D)^{\frac 12} f_N \|_{L^p(\Omega)}\lesssim
\| (-\Delta_D)^{\frac 12} f \|_{L^p(\Omega)}\,.
$$
Since $\|f'_N-f'\|_{L^2(\Omega)}\rightarrow 0$, for some subsequence $f'_{N_j}(r)\rightarrow f'(r)$ pointwise a.e. By Fatou's lemma, it thus suffices to
prove \eqref{otherway} for $f\in\dot H^1_0$ with compact spectral support in $[0,\infty)$.
Such functions are smooth, as is $(-\Delta_D)^{\frac 12}f$.

For such $f$, we can write
\begin{align*}
f'(r) & = \frac 2\pi\int_0^\infty \frac \partial {\partial r} \left( \frac {\sin \lambda(r-1) } {\l r} \right) \F_0 \Bigl(\bigl(-\Delta_D\bigr)^{\frac 12} f\Bigr)(\l)\, d\lambda \\
& = \frac 2\pi\int_0^\infty
\left(\frac {\cos \lambda(r-1)} {r} - \frac{\sin \lambda(r-1)}{\l\, r^2}
\right) \F_0 \Bigl(\bigl(-\Delta_D\bigr)^{\frac 12} f\Bigr)(\l)\, d\lambda
\end{align*}
This can in turn be written as
$$
\frac 2\pi
\int_0^\infty K^T_1(r,s)\,\Bigl(\bigl(-\Delta_D\bigr)^{\frac 12} f\Bigr)(s)\,s^2ds\,,
$$
where $K_1^T$ is the transpose of the above kernel $K_1$,
$$
K^T_1(r,s) = \frac 1 {rs} \Bigl( \frac 1 {r+s-2} - \frac 1 {r-s} \Bigr)
+ \frac 1 {r^2 s} \log \left| \frac {r-s} {r+s-2}\right|\,.
$$
Subtracting off $K_0^T(r,s)$ reduces matters to establishing bounds for the kernel
$$
K^T(r,s)=\frac{1}{rs} \left( \frac 1{r+s-2} -\frac 1{r+s} \right)
+ \frac 1 {r^2s} \log \left( \frac{r+s}{r+s-2} \right)\,.
$$
The first term is the same as above. The second term gives a bounded integral operator
for $1<p<3$, since its transpose is bounded on $\frac 32<p<\infty$.
\end{proof}

\section{Proof of Theorem \ref{thm:main}}
We begin by making the definition of the solution more precise. Let
$I$ be a finite time interval containing $0$.
As remarked above, \eqref{radial_emb} implies that radial $\dot H^1_0$ is
closed under multiplication, so by $\dot H^1_0$ boundedness of $\exp(it\Delta_D)$
we have
\begin{equation}\label{energy_bound}
\left\|\,\int_0^t e^{i(t-s)\Delta_D}|u|^4 u(s)\,ds \,\right\|_{L_t^{\infty}\dot
H_0^1(I\times \Omega)} \lsm |I|\cdot \|u\|^5_{L_t^{\infty}\dot H_0^1(I\times\Omega)}.
\end{equation}
Therefore, if $u\in C(I;\dot H_0^1(\Omega))$, then the inhomogeneous
term will  also be in $\dot H_0^1(\Omega)$. This motivates the
following

\begin{defn}[Solution] Denote $F(u)=|u|^4 u$.
A radial function $u: I \times \Omega \to \C$ on a non-empty time
interval $I \subset \R$ (possibly infinite or semi-infinite) is a
\emph{strong $\dot H_0^1(\Omega)$ solution} (or \emph{solution} for
short) to \eqref{nls} if it lies in the class $C^0_t \dot H_0^1(I
\times \Omega)$, and we have the Duhamel formula
\begin{align}\label{old duhamel}
u(t_1) = e^{i(t_1-t_0)\Delta_D} u(t_0) - i \int_{t_0}^{t_1}
e^{i(t_1-t)\Delta_D} F(u(t))\,dt
\end{align}
for all $t_0, t_1 \in I$. We refer to the interval $I$ as the
\emph{lifespan} of $u$. We say that $u$ is a \emph{maximal-lifespan
solution} if the solution cannot be extended to any strictly larger
interval. We say that $u$ is a \emph{global solution} if $I = \R$.
\end{defn}

Using \eqref{energy_bound} we can
easily construct the local solution of \eqref{nls} using a fixed
point argument in $C^0_t\dot H_0^1(\Omega)$.
Moreover, the lifespan of the local solution depends only on the
$\dot H_0^1(\Omega)$ norm of the initial data. Existence of the global
solution then follows quickly from the energy
conservation property of the defocusing equation \eqref{nls}.
Specifically, we have the following.

\begin{thm}[Global well-posedness]\label{thm:local}
Let $u_0\in \dot H_0^1(\Omega)$ be spherically symmetric. Then there
exists a unique global solution $u\in C(\R;\dot H_0^1(\Omega))$.
Moreover, $\nabla u\in L_{t, loc}^qL_x^r(\R\times\Omega)$ for any
admissible pair $(q,r)$, if $r<3$. For any $t\in\R$, we have
\begin{equation*}
E(u(t))=\frac 12\int_{\Omega}|\nabla u(t,x)|^2\,dx+\frac
16\int_{\Omega}|u(t,x)|^6 \,dx=E(u_0)\,.
\end{equation*}
For this global solution, scattering holds provided the global space-time $L^{10}$
norm is bounded. Precisely,
suppose that $u$ satisfies
$$
\|u\|_{L_{t,x}^{10}([0,\infty)\times \Omega)}<\infty\,.
$$
Then $u$ scatters forward in time, i.e.\ there exists unique $v_+\in \dot H_0^1(\Omega)$ such that
$$
\lim_{t\to \infty}\|e^{it\Delta_D} v_+-u(t)\|_{\dot H_0^1(\Omega)}=0\,.
$$
The same statement holds backward in time.
\end{thm}

\smallskip

The fact that global $L^{10}_{t,x}$ control of
the norm implies finiteness of Strichartz norms and scattering is established by similar
steps to those leading from \eqref{Strichest} to \eqref{small_strichartz} below.
Furthermore, a standard continuity argument shows that if
$\|u_0\|_{\dot H_0^1} < \epsilon$ for small $\epsilon$,
then the corresponding solution scatters in both
time directions. (see\cite{cw} fro instance).

Due to Theorem \ref{thm:local}, the proof
of Theorem \ref{thm:main} is reduced to showing that the $L_{t,x}^{10}$
norm of the solution over any compact time interval is bounded by a constant
depending only on upper bounds for the initial energy.
Theorem \ref{thm:main} is thus a consequence of the following.

\begin{thm}\label{thm:spacetime}
Assume $u\in \dot H_0^1(\Omega)$ is a spherically symmetric solution of \eqref{nls} on a
compact interval $[t_-,t_+]$. Suppose $E(u_0)\le E$. Then
$$
\|u\|_{L_{t,x}^{10}([t_-,t_+]\times \Omega)}<C(E).
$$
\end{thm}

The rest of this section will be devoted to the proof of Theorem \ref{thm:spacetime}. We begin with some useful conventions.

\textbf{Convention}. Let $0<\eta_3\ll\eta_2\ll\eta_1\ll\eta_0\ll 1$ be small constants to be determined. We use $c(\eta_i)$ to denote a small
constant depending on $\eta_i$ such that $\eta_{i+1}\ll c(\eta_i)\ll \eta_i$.
We use $C(\eta_i)$ to denote a large constant such that $\frac 1{\eta_i}\ll C(\eta_i)\ll \frac 1{\eta_{i+1}}.$
The constants $c(\eta_i)$ and $C(\eta_i)$ will sometimes vary from line to line, but the dependence is clear
from the context.
The notation $a\lesssim b$ will be used to mean that $a\le C(E)\,b$, where $C(E)$ may depend
on the energy upper-bound $E$.

We will use $\phi(x)$ to denote a radial smooth cutoff function such that
\begin{align} \label{def_phi}
\phi(x)=
\begin{cases}
1, \quad \text{if $|x|\le 1$}\,, \\
0, \quad \text{if $|x|>2$}\,.
\end{cases}
\end{align}

We also denote $\phi_{<C}(x)=\phi(\frac xC)$, $\phi_{>C}=1-\phi_{<C}$.

\vspace{0.2cm}

Since $\|u\|_{L^\infty_tL_x^{10}([t_-,t_+]\times\Omega)}\lsm E$, we decompose
$$
[t_-,t_+]=\bigcup_{j=1}^{J} I_j
$$
such that
$$
\eta_0<\|u\|_{L_{t,x}^{10}(I_j\times \Omega)}\le 2\eta_0.
$$
By Strichartz estimates on $I_j$, we have
\begin{align}\notag
\|\nabla u\|_{L_t^8 L_x^{\frac{12}5}(I_j\times\Omega)}&\lesssim
\|u(t_i)\|_{\dot H_0^1(\Omega)}
+\|\nabla (|u|^4 u)\|_{L_t^{\frac{40}{21}}L_x^{\frac{60}{49}}(I_j\times\Omega)}\\ \label{Strichest}
&\lesssim 1+\|u\|_{L_{t,x}^{10}(I_j\times\Omega)}^4
\|\nabla u\|_{L_t^8 L_x^{\frac{12}5}(I_j\times\Omega)}\\ \notag
&\lesssim 1+\eta_0^4\|\nabla u\|_{L_t^8 L_x^{\frac{12}5}(I_j\times\Omega)}\,.
\end{align}

Note we have used Proposition \ref{sobolev} to deduce the first inequality. By taking
$\eta_0$ small we have
$$
\|\nabla u\|_{L_t^8 L_x^{\frac{12}5}(I_j\times\Omega)}\lsm 1\,.
$$
A further application of Strichartz estimates yields that, for admissible pairs $(q,r)$
with $r<3$,
\begin{align}\label{small_strichartz}
\|\nabla u\|_{L_t^qL_x^r(I_j\times\Omega)}\lesssim 1\,.
\end{align}

Now let $u_+(t)=e^{i(t-t_+)\Delta_D} u(t_+),\  u_-(t)=e^{i(t-t_-)\Delta_D}u(t_-)$. We distinguish between two cases for each $I_j$:

$\bullet$ $I_j$ is called exceptional if either
\begin{align*}
&\|u_+\|_{L_{t,x}^{10}(I_j\times \Omega)}>\eta_0^{10}\;\;\;\text{or}\\
&\|u_-\|_{L_{t,x}^{10}(I_j\times \Omega)}>\eta_0^{10}\,,
\end{align*}

$\bullet$ $I_j$ is called unexceptional if
$$
\|u_{\pm}\|_{L_{t,x}^{10}(I_j\times \Omega)} \le \eta_0^{10}\,.
$$
Since
$\|u_\pm\|_{L^{10}_{t,x}([t_-,t_+])\times\Omega)}\lesssim
\|\nabla u_\pm\|_{L^{10}_tL^{\frac{30}{13}}_x([t_-,t_+])\times\Omega)}
\lesssim \|u(t_\pm)\|_{\dot{H^1_0}}\,,$ the number of exceptional intervals is bounded by
$C(\eta_0,E)$. We thus need to control only the number of unexceptional intervals.

We first prove that the length of unexceptional intervals has a uniform lower bound.
This is in contrast to the whole-space case, where $|I_j|$
may be arbitrarily small.

\begin{lem} \label{lem:lowerbound}
Let $I$ be an unexceptional interval, then
$$
|I|>\eta_1.
$$
\end{lem}
\begin{proof}
Denote
$$
I=[a,b].
$$
Without loss of generality, we assume\footnote{Otherwise if this holds for
$[a,\frac{a+b}2]$, and we apply a similar argument by just reversing the time direction.}
$$
\|u\|_{L_{t,x}^{10}([\frac {a+b}2,b]\times \Omega)}\ge \frac {\eta_0}2.
$$
By the Duhamel formula
$$
u(t)=u_-(t)-i\int_{t_-}^a e^{i(t-s)\Delta_D}F(u)(s)ds-i\int_a^t e^{i(t-s)\Delta_D} F(u)(s)ds
\,.
$$
We define
\begin{align*}
w(t):&=i\int_{t_-}^a e^{i(t-s)\Delta_D}F(u)(s)ds \\
&=-u(t)+u_-(t)-i\int_a^t e^{i(t-s)\Delta_D}F(u)(s)ds\,.
\end{align*}
We next observe that $w$ has certain bounds. By Strichartz estimates and the steps
leading from \eqref{Strichest} to \eqref{small_strichartz},
$$
\sup_t\|w(t,\cdot\,)\|_{\dot H_0^1(\Omega)}\lsm 1\,, \quad \forall\, t\in I.
$$
Moreoever, we have
\begin{align}\notag
\|u_-\|_{L_{t,x}^{10}([\frac {a+b}2,b]\times \Omega)}&\le \eta_0^{10}\,,\\
\label{intFest}
\left\|\,\int_a^te^{i(t-s)\Delta_D}F(u)(s)ds\,
\right\|_{L_{t,x}^{10}([\frac {a+b}2,b]\times\Omega)}&\lesssim \eta_0^4\le \eta_0^2\,.
\end{align}
By the triangle inequality we then have
\begin{equation}\label{w-ten}
\|w\|_{L_{t,x}^{10}([\frac {a+b}2,b]\times \Omega)}\ge \frac {\eta_0}4.
\end{equation}

\vspace{0.2cm}

We next consider an upper bound on the $L^{10}$ norm of $w$.
For the low frequency component, we use the Bernstein inequality
\begin{align} \label{w-low}
\| P_{<c(\eta_0) |I|^{-\frac 12}} &w \|_{L_{t,x}^{10}([\frac{a+b}2,
b]\times \Omega)} \notag \\
\lesssim & \,|I|^{\frac 1{10}} (c(\eta_0) |I|^{-\frac 12})^{3(\frac 16-\frac 1{10})} \|w\|_{L_t^\infty L_x^6([\frac{a+b}2, b] \times \Omega)}\\
\lesssim & \,c(\eta_0)^{\frac 15}
\le \eta_0^2\,.
\notag
\end{align}

For the high frequency component, we use dispersive estimates to obtain
\begin{align}\notag
\bigl\|P_{>c(\eta_0)|I|^{-\frac 12}}w
 &\bigr\|_{L_{t,x}^{10}([\frac
{a+b}2,b]\times \Omega)}
\notag\\
& =\biggl\|\,\int_{t_-}^a e^{i(t-s)\Delta_D}
P_{>c(\eta_0)|I|^{-\frac 12}} F(u(s,\cdot\,))ds\,
\biggr\|_{L_{t,x}^{10}([\frac {a+b}2,b]\times \Omega)}\\
\notag
&\lesssim |I|^{\frac 1{10}}\int_{t_-}^a |\tfrac{a+b}2-s|^{-\frac 65}
\bigl\|P_{>c(\eta_0)|I|^{-\frac 12}}
F(u(s,\cdot\,))\bigr\|_{L_x^{\frac {10}9}(\Omega)}ds\\
\label{highfreqest}
&\lesssim |I|^{\frac 1{10}}|I|^{-\frac 15} (c(\eta_0)|I|^{-\frac 12})^{-1}
\|\nabla F(u)\|_{L_t^{\infty}L_x^{\frac {10}9}([t_-,a]\times\Omega)}\\
\notag
&\lesssim |I|^{-\frac 1{10}}C(\eta_0)\,|I|^{\frac 12}
\|\nabla u \|_{L_t^{\infty}L_x^2([t_-,t_+]\times \Omega)}\,\|u\|^4_{L_t^{\infty}L_x^{10}([t_-,t_+]\times \Omega)}\,.
\end{align}
Using the radial Sobolev embedding \eqref{radial_emb} and the
interpolation bound
\begin{align*}
\|u\|_{L_t^{\infty}L_x^{10}([t_-,t_+]\times \Omega)}\lesssim
\|u\|_{L_t^{\infty}L_x^6([t_-,t_+]\times\Omega)}^{\frac 35}
\|u\|_{L_t^{\infty}L_x^{\infty}([t_-,t_+]\times \Omega)}^{\frac 25}\lsm 1\,,
\end{align*}
we have the bound
$$
\bigl\|P_{>c(\eta_0)|I|^{-\frac 12}}w\bigr\|_{L_{t,x}^{10}([\frac {a+b}2,b]\times\Omega)}
\lesssim C(\eta_0)|I|^{\frac 25}.
$$
If $|I|\le\eta_1$, by taking $\eta_1$ sufficiently small we have
\begin{equation}\label{w-high}
\bigl\|P_{>c(\eta_0)|I|^{-\frac 12}}w\bigr\|_{L_{t,x}^{10}([\frac {a+b}2,b]\times \Omega)}
\le \frac {\eta_0}{100}\, .
\end{equation}
The estimate of \eqref{w-low}, \eqref{w-high} together imply that
$$
\|w\|_{L_{t,x}^{10}([\frac {a+b}2,b]\times \Omega)}\le \frac{\eta_0}{10}\,,
$$
which contradicts \eqref{w-ten}. We conclude that
$$
|I|>\eta_1\,,
$$
if $I$ is an unexceptional interval.
\end{proof}

Another important property about unexceptional intervals is that the solution will
concentrate on them due to the nontrivial spacetime bound.
We begin with local mass conservation.

Let $\phi$ be the smooth cutoff function defined in \eqref{def_phi}, and
$\phi_R(x)=\phi(x/R)$. Define the
local mass of $u$ to be
$$
M_R(t)=\int_{|x|\ge 1} |u(t,x)|^2 \phi_R^2(x)\,dx\,.
$$
Then we have
\begin{lem}[Local mass conservation]
\begin{align}\label{mass_dev}
\Bigl|\frac d{dt}M_R^{\frac12}(t)\Bigr|\lesssim \frac 1 R\,.
\end{align}
\end{lem}
\begin{proof}Using the equation in \eqref{nls}, we compute
\begin{align*}
\frac d{dt}M_R(t)&=2 \textit{Re}\int_{|x|\ge 1}u_t\,\bar u \,\phi_R^2(x)\,dx
\\
&=-2\,\textit{Im} \int_{|x|\ge 1} \Delta u \cdot \bar u \,\phi_R^2(x) \,dx \\
&=\frac 2R \,\textit{Im}\int_{|x|\ge 1} \nabla u \,\bar u \,\phi_R(x)
(\nabla \phi)_R(x)\,dx\,.
\end{align*}
Therefore
\begin{align*}
\Bigl|\frac d{dt}M_R(t)\Bigr|&\lesssim \frac 2R \,\|\nabla u\|_{L^2(|x|\ge 1)}\,
\|u\,\phi_R\|_{L^2(|x|\ge 1)}\\
&\lesssim \frac 1R M_R^{\frac 12}(t)\,.
\end{align*}
From here, \eqref{mass_dev} follows directly.
\end{proof}

We next establish the important

\begin{lem}[Mass concentration on unexceptional intervals]
Let $I$ be an unexceptional interval. Then for any $t\in I$,
$$
\int_{|x|<\frac 1{\eta_3}|I|^{\frac 12}}|u(t,x)|^2dx\ge c(\eta_2) |I|\,.
$$
\end{lem}
\begin{proof}
We assume $I=[a,b]$ and make the same simplifications as in the beginning part of the proof of Lemma
\ref{lem:lowerbound}.
We first prove the $L_{t,x}^{10}$-norm of the high frequency part of the solution is negligible. A calculation similar to \eqref{highfreqest} yields
\begin{equation*}
\bigl\|P_{>C(\eta_2)|I|^{-\frac 12}}w\bigr\|_{L_{t,x}^{10}([\frac {a+b}2,b]\times \Omega)}
\lesssim |I|^{-\frac 1{10}} \bigl\|P_{>C(\eta_2)|I|^{-\frac 12}}F(u)
\bigr\|_{L_t^{\infty}L_x^{\frac{10}9}([t_-,t_+]\times \Omega)}\,.
\end{equation*}
First consider the case that $\eta_2|I|^{\frac 12}\ge 2$.
We estimate the norm of $F(u)$ in different spatial regimes.
First,
\begin{align*}
\bigl\|P_{>C(\eta_2)|I|^{-\frac 12}} &\phi_{<\eta_2|I|^{\frac 12}}F(u)
\bigr\|_{L_x^{\frac {10}9}(\Omega)} \\
&\lesssim \|\phi_{<\eta_2  |I|^{\frac 12}}F(u)\|_{L_x^{\frac {10}9}(\Omega)} \\
&\lesssim \|F(u)\|_{L_x^{\frac 65}(\Omega)}\|\phi_{<\eta_2 |I|^{\frac 12}}\|_{L_x^{15}(\Omega)} \\
&\lesssim \eta_2^{\frac 15}|I|^{\frac 1{10}} \|u\|_{L_x^6(\Omega)}^5\\
&\lesssim \eta_2^{\frac 15}|I|^{\frac 1{10}}\,.
\end{align*}
Next, using the radial Sobolev embedding
\begin{align*}
\||x|^{\frac 15}u\|_{L_x^{10}}\lsm \|\nabla u\|_{L_x^2}
\end{align*}
together with the
Bernstein inequality and Proposition \ref{sobolev}, we estimate
\begin{align*}
\bigl\|&P_{>C(\eta_2)|I|^{-\frac 12}}\bigl(\phi_{>\eta_2 |I|^{\frac 12}}F(u)\bigr)
\bigr\|_{L_x^{\frac {10}9}(\Omega)} \\
&\lesssim c(\eta_2)|I|^{\frac 12} \bigl\|\nabla \bigl(\phi_{>\eta_2|I|^{\frac 12}}F(u)\bigr)
\bigr\|_{L_x^{\frac {10}9}(\Omega)}\\
&\lesssim c(\eta_2)|I|^{\frac 12} \biggl(\|\nabla \phi_{>\eta_2 |I|^{\frac 12}}
\|_{L_x^{15}(\Omega)}\|u\|_{L_x^6(\Omega)}^5
+\|\phi_{>\eta_2|I|^{\frac 12}}u\|_{L_x^{10}(\Omega)}^4\|\nabla u\|_{L_x^2(\Omega)}\biggr)\\
&\lesssim c(\eta_2) |I|^{\frac 12}\bigl(\eta_2|I|^{\frac 12}\bigr)^{-\frac 45}
\bigl(\,\|u\|_{L_x^6(\Omega)}+\|\nabla u\|_{L_x^2(\Omega)}\bigr)^5 \\
&\lesssim c(\eta_2)\,\eta_2^{-\frac 45}|I|^{\frac 1{10}}
\le \eta_2^2 \,|I|^{\frac 1{10}}\,.
\end{align*}
Therefore
$$
\|P_{>C(\eta_2)|I|^{-\frac 12}}w\|_{L_{t,x}^{10}([\frac {a+b}2,b]\times\Omega)}\le \frac {\eta_0}{100}\,.
$$
In case $\eta_2|I|^{\frac 12}\le 2$, applying a similar argument without the cutoff $\phi$
yields the better bound $c(\eta_2)|I|^{\frac 12}$.

This combines with \eqref{w-ten}  give
\begin{equation}\label{medium-w}
\|P_{<C(\eta_2)|I|^{-\frac 12}}w\|_{L_{t,x}^{10}([\frac {a+b}2,b]\times \Omega)} \ge \frac {\eta_0}8\,.
\end{equation}
Recalling the definition of $w$, the boundedness of $P$, \eqref{intFest}
and \eqref{w-ten}, and the condition for unexceptional intervals, we have
\begin{equation}\label{medium-u}
\|P_{<C(\eta_2)|I|^{-\frac 12}}u\|_{L_{t,x}^{10}([\frac {a+b}2,b]\times\Omega)}\ge \frac {\eta_0}{10}\,.
\end{equation}
On the other hand, interpolation and the lower bound for $|I|$ yield
\begin{align*}
\|\phi_{>\frac 1{\eta_3} |I|^{\frac 12}}u&\|_{L_{t,x}^{10}([\frac {a+b}2,b]\times \Omega)} \\
&\lesssim |I|^{\frac 1{10}} \|\phi_{>\frac 1{\eta_3}|I|^{\frac 12}}u\|_{L_t^{\infty}L_x^{10}([\frac {a+b}2,b]\times \Omega)} \\
&\lesssim |I|^{\frac 1{10}} \|\phi_{>\frac 1{\eta_3} |I|^{\frac 12}}u\|_{L_t^{\infty}L_x^6([\frac {a+b}2,b]\times \Omega)}^{\frac 35} \|\phi_{>\frac 1{\eta_3}
|I|^{\frac 12}}u\|_{L_t^{\infty}L_x^{\infty}([\frac {a+b}2,b]\times \Omega)}^{\frac 25}
\\
&\lesssim |I|^{\frac 1{10}} \Bigl(\frac 1{\eta_3}|I|^{\frac 12}\Bigr)^{-\frac 15}\\
&\le \eta_0^2\,.
\end{align*}
Thus, \eqref{medium-u} can be improved to
$$
\bigl\|P_{<C(\eta_2)|I|^{-\frac 12}}\phi_{<\frac 1{\eta_3}|I|^{\frac 12}}u
\bigr\|_{L_{t,x}^{10}([\frac{a+b}2, b]\times\Omega)}\ge \frac {\eta_0}{20}\,.
$$
From this, the mass concentration follows quickly. Indeed, using the Bernstein and H\"older
inequalities yields
\begin{align*}
 \frac {\eta_0}{20}&\le
 \bigl\|P_{<C(\eta_2)|I|^{-\frac 12}}\phi_{<\frac 1{\eta_3}|I|^{\frac 12}}u
 \bigr\|_{L_{t,x}^{10}([\frac {a+b}2,b]\times \Omega)} \\
&\lsm |I|^{\frac 1{10}}\bigl(C(\eta_2)|I|^{-\frac 12}\bigr)^{3(\frac 12-\frac 1{10})}
\|\phi_{<\frac 1{\eta_3}
|I|^{\frac 12}} u\|_{L_t^{\infty}L_x^2([\frac {a+b}2,b]\times \Omega)}\\
&\lesssim C(\eta_2)|I|^{-\frac 12} \|\phi_{<\frac 1{\eta_3}|I|^{\frac 12}}u
\|_{L_t^{\infty}L_x^2([\frac {a+b}2,b]\times\Omega)}\,.
\end{align*}
Thus, there exists $t_0\in I$ such that
$$
\|u(t_0)\|_{L^2(1\le |x|<\frac 1{\eta_3}|I|^{\frac 12})}>c(\eta_2)|I|^{\frac 12}\,.
$$
Using \eqref{mass_dev} we get
$$
\|u(t)\|_{L^2(1\le |x|<\frac 1{\eta_3}|I|^{\frac 12})}>c(\eta_2)|I|^{\frac 12}\,,
\quad \forall\  t\in I\,.
$$
\end{proof}

\begin{prop}[Morawetz inequality] \label{prop37}
Let $I$ be a time interval, and let $A\ge 1$. Then
$$
\int_I \int_{1\le |x|\le A|I|^{\frac 12}} \frac {|u(t,x)|^6}{|x|} \,dx \,dt
\lsm A\,|I|^{\frac 12}\,.
$$
\end{prop}
\begin{proof}
We begin with the local momentum conservation identity
\begin{equation}\label{moment}
\partial_t \textit{Im}(\partial_k u \,\bar u)=
-2\partial_j \textit{Re}(\partial_k u\,\partial_j \bar u)+\frac 12 \partial_k \Delta (|u|^2)
-\frac 23 \partial_k|u|^6\,.
\end{equation}

Let $a(x)= |x| \phi_{<R}(x)\,,$ so that $a$ is a radial function.  Let
$a_{jk}=\partial_j\partial_k a$.
Observe that for $1\le |x|\le R\,,$
$$
a_{jk}(x) \;\text{is positive definite}\,,\quad\nabla a(x)=\frac x{|x|}\,,
\quad
\Delta^2 a(x)<0\,.
$$
In the region $|x|\ge R\,,$ $a(x)$ has the rough bound
$$
|\partial_k a|\lesssim 1\,, \quad |a_{jk}|\lesssim \frac 1R\,,\quad
|\Delta^2 a|\lsm \frac 1{R^3}\,.
$$

We multiply the first term in \eqref{moment} by $\partial_k a$ and integrate over
$\Omega$ to obtain
\begin{align*}
-2\int_{|x|\ge 1}
\partial_j \textit{Re}(\partial_k u\,&\partial_j \bar u)\,\partial_k a \,dx \\
&=2\int_{|x|\ge 1}\textit{Re}(\partial_k u\,\partial_j \bar u)\,a_{jk}\, dx
+2\int_{|x|=1}\textit{Re}(\partial_k u\,\partial_j \bar u)\,x_j\,x_k \,d\sigma\\
&=2\int_{1\le |x|\le R} \textit{Re}(\partial_k u\,\partial_j \bar u) \,a_{jk}\, dx
+2\int_{|x|\ge R}\textit{Re}(\partial_k u\,\partial_j \bar u)\, a_{jk}\, dx\\
&\qquad\qquad +2\int_{|x|=1}|\nabla u|^2 \,d\sigma(x) \\
&\ge -\frac 1R \|\nabla u\|_{L^2(\Omega)}^2+2\int_{|x|=1}|\nabla u|^2
d\sigma(x)\,.
\end{align*}
We do the same for the second term in \eqref{moment}, and use the Dirichlet condition
to calculate
\begin{align*}
\frac 12 \int_{|x|\ge 1}\partial_k &\Delta (|u|^2)\,\partial_k a\,
dx\\
&=-\frac 12 \int_{|x|\ge 1} \Delta(|u|^2) \,\Delta a \, dx
-\frac 12 \int_{|x|=1}\Delta (|u|^2)\, d\sigma \\
&=-\frac 12 \int_{|x|\ge 1}|u|^2 \Delta^2 a \, dx
-\int_{|x|=1}|\nabla u|^2 \, d\sigma(x)\\
&=-\frac 12 \int_{1\le |x|\le R}|u|^2\Delta^2 a \, dx
-\frac 12\int_{|x|\ge R}|u|^2 \Delta^2 a \, dx-\int_{|x|=1}|\nabla u|^2 \,d\sigma(x)\\
&\ge -\frac CR \|u\|^2_6 -\int_{|x|=1}|\nabla u|^2 \,d\sigma(x)\,.
\end{align*}
Similarly for the third term in \eqref{moment},
\begin{align*}
-\frac 23\int_{|x|\ge1}\partial_k(|u|^6)\,\partial_k a \, dx
&=\frac 23
\int_{|x|\ge 1} |u|^6 \Delta a \, dx\\
&=\frac 43\int_{1\le |x|\le R} \frac {|u|^6}{|x|} \,dx +\frac 23
\int_{|x|\ge R} |u|^6 \Delta a \, dx\\
&\ge \frac 43 \int_{1\le |x|\le R}\frac {|u|^6}{|x|} \,dx-\frac 1R
\|u\|_6^6\,.
\end{align*}
Notice also that
$$
\Bigl|\int_{|x|\ge 1} \textit{Im}(\partial_k u\,\bar u)\,\partial_k a\, dx\Bigr|\lesssim
\|\partial_k u\|_{L^2(\Omega)} \|u\|_{L^6(\Omega)} \|\partial_k
a\|_{L^3(\Omega)}\lesssim R\,.
$$
Integrating \eqref{moment} over $I\times \Omega$ we get
$$
\int_I\int_{1\le |x|\le R}\frac {|u|^6}{|x|} \,dx \, dt\lesssim
\frac{|I|}R+R\,.
$$
Taking $R=A\,|I|^{\frac 12}$, since $A\ge 1$ we have
$$
\int_I\int_{1\le |x|\le A|I|^{\frac 12}}\frac {|u|^6}{|x|} \,dx \,dt\lsm
A\,|I|^{\frac 12}\,.
$$
\end{proof}

\begin{lem}
Let $J$ be an interval that contains a contiguous collection
$\bigcup_j I_j$ of unexceptional intervals. Then we have
\begin{equation}\label{ij}
\sum |I_j|^{\frac 12} \le C(\eta_2,\eta_3)|J|^{\frac 12}\,.
\end{equation}
\end{lem}
\begin{proof}
We apply mass concentration on each of the time intervals $I_j$ to get
$$
c(\eta_2)|I_j|\le \int_{1\le |x|\le \frac 1{\eta_3}|I_j|^{\frac 12}}
|u(t,x)|^2\, dx \lesssim
\biggl(\int_{1\le |x|\le\frac 1{\eta_3}|I_j|^{\frac 12}}
\frac {|u(t,x)|^6}{|x|}\, dx\biggr)^{\frac 13}
\Bigl(\frac 1{\eta_3}|I_j|^{\frac 12}\Bigr)^{\frac 73}\,.
$$
Therefore
$$
c(\eta_2,\eta_3)|I_j|^{-\frac 12}\lesssim \int_{1\le |x|\le \frac
1{\eta_3}|I_j|^{\frac 12}}\frac {|u(t,x)|^6}{|x|}\, dx\,.
$$
We integrate in time over $I_j$ and sum over $j$. The Morawetz inequality then gives
$$
c(\eta_2,\eta_3)\sum |I_j|^{\frac 12}\lesssim \frac
1{\eta_3}|J|^{\frac 12}\,,
$$
which implies \eqref{ij}\,.
\end{proof}

At this point, we can repeat an argument of Bourgain \cite{borg:scatter},
to get the upper bound of the number of unexceptional intervals. We
record the result without repeating the proof.
\begin{thm}
There exists $C(E,\eta_0,\eta_1,\eta_2,\eta_3)$ such that
\begin{align*}
\#\{I_j, I_j \mbox{ is unexceptional}\}\le C(E, \eta_0, \eta_1, \eta_2, \eta_3)\,.
\end{align*}
\end{thm}
This, combined with the fact that the number of exceptional intervals is finite, proves Theorem \ref{thm:main}.

\section{The energy supercritical problem}

For spherically symmetric function $f\in \dot H_0^1(\Omega)$, the bound
\begin{align*}
\|f\|_{L_x^{\infty}(\Omega)}\le \||x|^{\frac 12}f\|_{L^\infty_x(\Omega)}
\lesssim \|f\|_{\dot H_0^1(\Omega)}\,,
\end{align*}
means that
any supercritical nonlinearity $|u|^p u$ for any $p>4$ can still be
viewed as ``critical". As a consequence, the proof of Theorem
\ref{thm:main} can be applied to the supercritical case after some
minor modifications. More specifically, we  consider the following
energy supercritical NLS in $\Omega$ with the Dirichlet boundary
condition:
\begin{align}\label{nls_super}
\begin{cases}
i\partial_t u+\Delta u=|u|^p u\,, \ p>4\,,\\
u(t,x)|_{\mathbb R\times\partial \Omega}=0\,,\\
u(0,x)=u_0(x).
\end{cases}
\end{align}
We then have the following result.

\begin{thm} Let $u_0\in \dot H_0^1(\Omega)$ be spherically symmetric. Then there exists
a unique solution $u\in C(\mathbb R; \dot H_0^1)$ to
\eqref{nls_super}, and for this solution it holds that
\begin{align*}
&\mbox{Energy  Conservation}: \\
&\qquad E(u(t))=\frac 12 \|\nabla u(t)\|_{L_x^2(\Omega)}^2+\frac 1{p+2}\|u(t)\|_{L_x^{p+2}(\Omega)}^{p+2}=E(u_0)\,,\\
&\mbox{Global spacetime Bound}: \rule{0pt}{12pt}\\
&\qquad \rule{0pt}{12pt} \|u\|_{L_{t,x}^{10}(\mathbb R\times \Omega)}\le
C\bigl(\|u_0\|_{\dot H_0^1(\Omega)}\bigr)\,.
\end{align*}
Moreover, there exist unique $v_{\pm}\in \dot H_0^1(\Omega)$ such that
\begin{align*}
\lim_{t\to \pm\infty}\|u(t)-e^{it\Delta_D}v_{\pm}\|_{\dot H_0^1(\Omega)}=0\,.
\end{align*}

\end{thm}

\section{The case with inhomogeneous $H_0^1(\Omega)$ data}\label{sec_inhomo}

In this section we point out that for the energy critical
problem,  Bourgain's argument is not needed if we assume inhomogeneous data $u_0 \in H_0^1(\Omega)$.

Indeed, consider \eqref{nls} with radial $u_0 \in H_0^1(\Omega)$.
By taking $a(x)=|x|$ and using almost the same computation
as in Proposition \ref{prop37}, we arrive at
\begin{align*}
\left\| |x|^{-\frac 16} u \right\|^6_{L_{t,x}^6 (\mathbb R\times
\Omega)} \lesssim \|u_0\|_{H^1_0(\Omega)}^2\,.
\end{align*}
This interpolates with the radial Sobolev embedding
\eqref{radial_emb}  immediately yield
\begin{align*}
\|u \|_{L_{t,x}^{10} (\mathbb R\times \Omega)} \lesssim_u 1\,,
\end{align*}
which is enough to prove scattering.

Finally, the same argument with small changes in numerology applies
to the energy supercritical case \eqref{nls_super}. We omit the
details.

\end{document}